\documentclass[11pt]{article}
\usepackage{amsmath}
\usepackage{amsthm}
\usepackage{amssymb,amsfonts}
\usepackage{hyperref}
\usepackage{latexsym}
\usepackage{todonotes}
\usepackage{nicefrac}
\usepackage{epsfig}
\usepackage{stmaryrd}
\usepackage{setspace}
\usepackage{tikz-cd}
\usepackage{enumerate}
\usepackage[all]{xypic}
\usepackage{bbm,ifpdf,tikz}
\ifpdf
\usepackage{pdfsync}
\fi

\newtheorem{theorem}{Theorem}[section]

\newtheorem{lemma}[theorem]{Lemma}
\newtheorem{proposition}[theorem]{Proposition}
\newtheorem{corollary}[theorem]{Corollary}

\newtheorem{conjecture}[theorem]{Conjecture}

{
\theoremstyle{definition}

\newtheorem{remark}[theorem]{Remark}
}

\newcommand{\excise}[1]{}

\newcommand{\rk}{\operatorname{rk}}

\renewcommand{\and}{\qquad\text{and}\qquad}
\newcommand{\Ind}{\operatorname{Ind}}
\newcommand{\Res}{\operatorname{Res}}

\renewcommand{\cH}{\mathcal{H}}

\newcommand{\la}{\lambda}

\newcommand{\tM}{\tilde{M}}
\newcommand{\fS}{\mathfrak{S}}
\newcommand{\WT}{W_{\!T}}
\newcommand{\WS}{W_{\!S}}

\newcommand{\WJ}{W_{\!J}}
\newcommand{\WH}{W_{\!H}}
\newcommand{\beq}{\begin{eqnarray*}}
\newcommand{\eeq}{\end{eqnarray*}}

\begin{document}
\spacing{1.2}
\noindent{\Large\bf Equivariant Kazhdan--Lusztig theory of paving matroids}\\

\noindent{\bf Trevor Karn}\\
School of Mathematics, University of Minnesota, Minneapolis, MN 55455
\vspace{.1in}

\noindent{\bf George Nasr}\footnote{Supported by NSF grant DMS-2053243.}\\
Department of Mathematics, University of Oregon, Eugene, OR 97403
\vspace{.1in}

\noindent{\bf Nicholas Proudfoot}\footnote{Supported by NSF grants DMS-1954050, DMS-2039316, and DMS-2053243.}\\
Department of Mathematics, University of Oregon, Eugene, OR 97403
\vspace{.1in}

\noindent{\bf Lorenzo Vecchi}\footnote{Partially supported by the National Group for Algebraic and Geometric Structures, and their Applications (GNSAGA - INdAM)}\\
Dipartimento di Matematica, Universit\`a di Bologna, 40126 Bologna BO, Italy\\

{\small
\begin{quote}
\noindent {\em Abstract.}
We study the way in which equivariant Kazhdan--Lusztig polynomials, equivariant inverse Kazhdan--Lusztig polynomials, and equivariant $Z$-polynomials
of matroids change under the operation of relaxation of a collection of stressed hyperplanes.  This allows us to compute these polynomials for arbitrary paving matroids,
which we do in a number of examples, including various matroids associated with Steiner systems that admit actions of Mathieu groups.
\end{quote} }

\section{Introduction}
We consider in this paper three polynomial invariants of matroids:  the {\bf Kazhdan--Lusztig polynomial} \cite{elias-proudfoot}, 
the {\bf inverse Kazhdan--Lusztig polynomial} \cite{gao-xie}, and
the {\bf \boldmath{$Z$}-polynomial} \cite{proudfoot-zeta}.
These invariants have been computed for uniform matroids \cite{gpy-equi-kl,kazhdan-uniform,gao-xie}, then generalized to sparse paving matroids 
\cite{lee-nasr-radcliffe-rho-removed,lee-nasr-radcliffe-sparse,ferroni-vecchi}, and finally to arbitrary paving matroids \cite{fnv2021}.

The purpose of this paper is to generalize the results of \cite{fnv2021} to the equivariant setting.  
Suppose that $M$ is a matroid and $W$ is a finite group that acts on the ground set of $M$, preserving the set of bases.
We then have the {\bf equivariant Kazhdan--Lusztig polynomial} \cite{gpy-equi-kl}, 
the {\bf equivariant inverse Kazhdan--Lusztig polynomial} \cite{proudfootkls}, and
the {\bf equivariant \boldmath{$Z$}-polynomial} \cite{proudfoot-zeta}; these are polynomials whose coefficients are isomorphism classes
of $W$-representations over the rational numbers, and they have the property that taking dimensions recovers the non-equivariant polynomials.
The equivariant Kazhdan--Lusztig polynomials and inverse Kazhdan--Lusztig polynomials have been computed for uniform matroids \cite{gpy-equi-kl,GXY21-inverse-equivariant}, 
and it is straightforward to use these results to compute the equivariant $Z$-polynomials.
Our main result extends these computations to paving matroids.

Our approach is based on the notion of {\bf relaxation}.  If $M$ is a matroid of rank $k$, a hyperplane $H$ of $M$ is called {\bf stressed} if every $k$-element subset of $H$
is a circuit.  In this case, there is a new matroid $\tilde M$ whose bases consist of all of the bases for $M$ along with all of the $k$-element subsets of $H$.
A matroid is paving if and only if it can be transformed into a uniform matroid by a sequence of relaxations.
The operation of relaxation changes our three polynomials in a controlled way, thus one may leverage the formulas for uniform matroids to obtain formulas for
paving matroids \cite{fnv2021}.  We will use the same idea in the equivariant setting, though now rather than relaxing one stressed hyperplane at a time, we will relax
one $W$-orbit of stressed hyperplanes at a time.

\begin{remark}
The notion of stressed hyperplanes and their relaxations is quite new, appearing for the first time in \cite{fnv2021}.  When $H$ is a circuit-hyperplane, it is a much more well-known
operation; see for example \cite[Proposition 1.5.14]{oxley}.  A matroid is sparse paving if and only if it can be transformed into a uniform matroid by a sequence of relaxations
of circuit-hyperplanes.
\end{remark}

\begin{remark}
It is conjectured that asymptotically almost all matroids are sparse paving \cite{mayhew}, and a logarithmic version of this conjecture has been proved \cite{pendavingh-vanderpol}.
In this sense, the results of this paper allow us to compute our three equivariant polynomials for most matroids.
\end{remark}

We now give a more precise statement of our results.  Given a matroid $M$ equipped with an action of a finite group $W$, we denote the equivariant Kazhdan--Lusztig polynomial
of $M$ by $P_M^W(t)$, the equivariant inverse Kazhdan--Lusztig polynomial of $M$ by $Q_M^W(t)$, and the equivariant $Z$-polynomial of $M$ by $Z_M^W(t)$.

Given a stressed hyperplane $H$, we write $\tilde M$ to denote the matroid obtained by simultaneously relaxing all hyperplanes in the $W$-orbit of $H$.  
Note that the action
of $W$ on $M$ induces an action on $\tilde M$.  Let $\WH$ be the stabilizer of $H$ in $W$.  
The group $\WH$ acts on $H$, inducing a homomorphism from $W$
to the permutation group $\fS_H$.  If $h= |H|$ and we fix an ordering of $H$, then we can identify $\fS_H$ with the symmetric group $\fS_h$.
For any representation $V$ of $\fS_h$,  we will write $\Res_{\WH}^{\fS_h} V$ to denote the pullback of $V$ to a representation of $\WH$
(even though the homomorphism from $\WH$ to $\fS_h$ need not be an inclusion).

\begin{theorem}\label{main}
Fix integers $h\geq k\geq 1$.  There exist polynomials 
$p_{k,h}^{\fS_h}(t)$, $q_{k,h}^{\fS_h}(t)$, and $z_{k,h}^{\fS_h}(t)$,
each with isomorphism classes of $\fS_h$-representations as coefficients, such that for 
any matroid $M$ of rank $k$, any group $W$ of symmetries of $M$, and any stressed hyperplane $H$ of cardinality $h$, the following identities hold:
\beq P_{\tM}^W(t) &=& P_{M}^W(t) + \Ind_{\WH}^{W}\Res_{\WH}^{\fS_h}p_{k,h}^{\fS_h}(t)\\
Q_{\tM}^W(t) &=& Q_{M}^W(t) + \Ind_{\WH}^{W}\Res_{\WH}^{\fS_h}q_{k,h}^{\fS_h}(t)\\
Z_{\tM}^W(t) &=& Z_{M}^W(t) + \Ind_{\WH}^{W}\Res_{\WH}^{\fS_h}z_{k,h}^{\fS_h}(t).\eeq
\end{theorem}

Next, we give explicit formulas for two of the three $\fS_h$-equivariant polynomials appearing in Theorem \ref{main}.
Given a partition $\la$ of $h$, we write $V_\la$ to denote the corresponding irreducible representation of $\fS_h$ over the rational numbers,
which is called the {\bf Specht module} associated with $\la$.  More generally, given a pair of partitions $\la$ and $\mu$ 
with $|\la| - |\mu| = h$,
we write $V_{\la/\mu}$ to denote the corresponding {\bf skew Specht module}, which is characterized by the property that the multiplicity of $V_\nu$ in $V_{\la/\mu}$
is equal to the multiplicity of $V_\la$ in $$\Ind_{\fS_{|\mu|}\times\fS_h}^{\fS_{|\la|}} \!\Big(V_\mu \boxtimes V_\nu\Big).$$
For uniform matroids, the coefficients of equivariant inverse Kazhdan--Lusztig polynomials are Specht modules \cite[Theorem 3.2]{GXY21-inverse-equivariant},
while the coefficients of equivariant Kazhdan--Lusztig polynomials are skew Specht modules \cite[Theorem 3.7]{GXY21-inverse-equivariant}.

\begin{theorem}\label{explicit}
When $k=1$, we have $p_{1,h}^{\fS_h}(t) = V_{[h]} = q_{1,h}^{\fS_h}(t)$ and $z_{1,h}^{\fS_h}(t) = 0$.  When $k>1$, 
we have the following explicit formulas:\footnote{In the expression for $q_{k,h}^{\fS_h}(t)$, 
we interpret the first term to be zero if $i=0$, and we interpret the second term to be zero if $i>0$ and $k=h$.}
\beq
p_{k,h}^{\fS_h}(t) &=& \sum_{0< i < k/2} V_{[h-2i+1,(k-2i+1)^i]/[k-2i,(k-2i-1)^{i-1}]}\ t^i\\
q_{k,h}^{\fS_h}(t) &=& \sum_{0\leq i < k/2} \left(V_{[h-k+2,2^{i-1},1^{k-2i}]} + V_{[h-k+1,2^i,1^{k-2i-1}]} \right)\, t^i.
\eeq
\end{theorem}

\begin{remark}
One can use similar methods to obtain an explicit formula for $z_{k,h}^{\fS_h}(t)$, but since this formula 
is considerably less elegant, we omit it.
\end{remark}

An unpublished conjecture of Gedeon states that the coefficients of the equivariant Kazhdan--Lusztig polynomial
of a matroid $M$ are bounded above by the coefficients of the equivariant Kazhdan--Lusztig polynomial of the uniform
matroid of the same rank on the same ground set.  We give a precise statement of this conjecture here; the non-equivariant version of the conjecture appears in
\cite[Conjecture 1.1]{lee-nasr-radcliffe-sparse}.

\begin{conjecture}\label{gedeon}
Let $M$ be a matroid of rank $k$ on the ground set $E$, and let $W$ be a finite group that acts on $E$ preserving $M$.
Then the coefficients of $P_{U_{k,E}}^W(t) - P_M^W(t)$ are honest (rather than virtual) representations of $W$.\footnote{We write $U_{k,E}$ to denote the uniform matroid
of rank $k$ on the set $E$, and $U_{k,n}$ to denote the uniform matroid of rank $k$ on the set $[n]$.  This differs from the notation in some of the references, where
$U_{m,d}$ is used to denote the uniform matroid of rank $d$ on the set $[m+d]$.}
\end{conjecture}

\begin{remark}
The fact that the constant and linear terms of $P_{U_{k,E}}^W(t) - P_M^W(t)$ are honest representations
follows from \cite[Corollary 2.10]{gpy-equi-kl}.  In higher degrees, the conjecture remains open.
\end{remark}

Theorems \ref{main} and \ref{explicit} imply that Conjecture \ref{gedeon} holds for paving matroids.

\begin{corollary}
Conjecture \ref{gedeon} holds when $M$ is paving.
\end{corollary}

\begin{proof}
If $M$ is paving, then $M$ may be transformed into $U_{k,E}$ by relaxing finitely many $W$-orbits of stressed hyperplanes.
Theorems \ref{main} and \ref{explicit} imply that each of these relaxations changes the equivariant Kazhdan--Lusztig polynomial
by adding a correction term whose coefficients are honest representations.
\end{proof}

Theorems \ref{main} and \ref{explicit}, along with the known formulas for uniform matroids, 
provide us with the tools to compute our equivariant polynomials for any paving matroid and any group of symmetries.
To illustrate this, we apply our results to compute the equivariant Kazhdan--Lusztig polynomials in six specific examples.  
First, we consider the V\'amos matroid, a sparse paving matroid of rank 4 with symmetry group isomorphic
to $D_4\times D_4$.  Our last five examples involve the Mathieu groups $M_{11}$, $M_{12}$, $M_{22}$, $M_{23}$, and $M_{24}$, which are sporadic finite simple groups.
Each of the Mathieu groups can be realized as a group of symmetries of a Steiner system, and therefore also of the paving matroid associated with that Steiner system.
The SageMath \cite{sagemath} code used to study each of these six examples is available at $$\textsf{https://github.com/trevorkarn/equivariant-matroid-relaxation.}$$

\vspace{\baselineskip}
\noindent
{\em Acknowledgments:}
The authors are grateful to June Huh suggesting the study of matroids associated with Steiner systems, and to Tom Braden and
Luis Ferroni for valuable discussions.

\section{Defining the polynomials}\label{sec:defs}
For any subset $S\subset E$, we write $M^S$ to denote the matroid obtained by localizing at $S$ (equivalently deleting the complement of $S$), and we 
write $M_S$ to denote the matroid obtained by contracting $S$.  Both of these matroids admit actions of the stabilizer group $\WS\subset W$. 
We denote the trivial representation of $W$ by $\tau_W$.

The equivariant Kazhdan--Lusztig polynomial $P_M^W(t)$ and the equivariant $Z$-polynomial $Z_M^W(t)$ 
are characterized by the following conditions:
\begin{itemize}
\item If the ground set of $M$ is empty, then $P_M^W(t) = Z_M^W(t) = \tau_W$.
\item If the ground set of $M$ is nonempty, then the degree of $P_M^W(t)$ is strictly smaller than half of the rank of $M$.
\item The polynomial $Z_M^W(t)$ is palindromic, with degree equal to the rank of $M$:
$$t^{\rk M} Z_M^W(t^{-1}) = Z_M^W(t).$$
\item For all $M$, \begin{equation}\label{PZ}Z_M^W(t)\; \; = \sum_{[S] \in 2^E/W} t^{\rk S} \Ind_{\WS}^S P_{M_S}^{\WS}(t).\end{equation}
\end{itemize}

\begin{remark}\label{R}
To be more explicit about how this works, let $M$ be a matroid of rank $k$ on a nonempty ground set $E$, and 
assume that equivariant Kazhdan--Lusztig polynomials have been defined for all matroids whose
ground sets are proper subsets of $E$.  Let
$$R_M^W(t) := \sum_{[S] \in (2^E\smallsetminus\{\emptyset\})/W} t^{\rk S} \Ind_{\WS}^S P_{M_S}^{\WS}(t).$$
Then $P_M^W(t)$ is the unique polynomial of degree strictly less than $k/2$ with the property that $Z_M^W(t) := P_M^W(t) + R_M^W(t)$ is palindromic of degree $k$.
\end{remark}

The equivariant inverse Kazhdan--Lusztig polynomial is characterized by the following two conditions \cite[Proposition 4.6]{proudfootkls}:
\begin{itemize}
\item If the ground set of $M$ is empty, then $Q_M^W(t) = \tau_W$.
\item If the ground set of $M$ is nonempty, then \begin{equation}\label{Q}\sum_{[S] \in 2^E/W} (-1)^{\rk S} \Ind_{\WS}^W\left(Q_{M^S}^{\WS}(t) \otimes P_{M_S}^{\WS}(t)\right) = 0.\end{equation}
\end{itemize}

\begin{remark}\label{loop}
The original definition of the (ordinary or equivariant) Kazhdan--Lusztig polynomial of $M$ and inverse Kazhdan--Lusztig polynomial of $M$ applied only to loopless matroids.
With this definition, one can prove inductively that $P_M^W(t)=0=Q_M^W(t)$ whenever $M$ has a loop.
In contrast, the polynomial $Z_M^W(t)$ are unchanged when we replace $M$ with its simplification.
These are the most natural definitions from the geometric point of view.
\end{remark}

\begin{remark}\label{subsets}
The contraction $M_S$ is loopless if and only if $S$ is a flat,
so Remark \ref{loop} implies that we may replace the sums in equations \eqref{PZ} and \eqref{Q} with sums over $W$ orbits in the lattice of flats of $M$.
However, it will be more convenient for our purposes to work with the sum over all subsets.
\end{remark}

\section{The first theorem}
This section is devoted to the proof of Theorem \ref{main}.  We begin with a technical lemma that will be a crucial ingredient in the proof.

Let $S$ be a nonempty proper subset of $H$.  Let $L(S,H) := \{w\in W\mid wS \subset H\}$.  This set admits an action by $\WH$ via left multiplication,
as well as a commuting action of $\WS$ via right multiplication.  The quotients by these actions can be described as follows:
\beq
C(S,H) &:=& \{wS\mid wS\subset H\} \cong L(S,H)/\WS\\
D(S,H) &:=& \{J\in \cH\mid S\subset J\}\cong \WH \backslash L(S,H).
\eeq
The double quotient $$\WH\backslash C(S,H) \cong \WS \backslash L(S,H)/\WH\cong D(S,H)/\WS$$ may be identified with the set of 
$\WS$-orbits of stressed hyperplanes that one must relax to go from $M_S$ to $\tM_S$.  Let $\widetilde{M_S}$ be the matroid obtained by relaxing only one of those orbits, namely the one containing the stressed hyperplane $H\smallsetminus S$. 

\begin{lemma}\label{claim1}
Suppose that Theorem \ref{main} holds for the matroids of rank equal to the rank of $M_S$.  Then we have
$$\Ind_{\WS}^W\left(P_{\tM_S}^{\WS}(t) -  P_{M_S}^{\WS}(t)\right)
= \sum_{[T]\in \WH\backslash C(S,H)}\Ind_{\WT}^W\left(P_{\widetilde{M_T}}^{\WT}(t) -  P_{M_T}^{\WT}(t)\right).$$
\end{lemma}

\begin{proof}
Let $i = |S|$.  Theorem \ref{main} for the action of $\WS$ on $M_S$ tells us that 
$$P_{\tM_S}^{\WS}(t) -  P_{M_S}^{\WS}(t) = \sum_{[J]\in D(S,H)/\WS} \Ind_{\WJ\cap\WS}^{\WJ}\Res^{\fS_{h-i}}_{\WJ\cap\WS} p_{k-i,h-i}^{\fS_{h-i}}(t).$$
Theorem \ref{main} for the action of $\WT$ on $M_T$ tells us that 
$$P_{\widetilde{M_T}}^{\WT}(t) -  P_{M_T}^{\WT}(t) = \Ind_{\WH\cap\WT}^{\WH}\Res^{\fS_{h-i}}_{\WH\cap\WT} p_{k-i,h-i}^{\fS_{h-i}}(t).$$
The lemma now follows from the identification of $D(S,H)/\WS$ with $\WH\backslash C(S,H)$.
\end{proof}

\begin{proof}[Proof of Theorem \ref{main}.]
We proceed by induction on $k$.  If $k=1$, then $H$ is necessarily the set of all loops in $M$, and $\WH=W$.
In this case, Remark \ref{loop} implies that we can take $p_{1,h}^{\fS_h}(t) = V_{[h]} = q_{1,h}^{\fS_h}(t)$ and $z_{1,h}^{\fS_h}(t) = 0$.

For the induction step, we will prove only the statements about $P_{\tM}^W(t)$ and $Z_{\tM}^W(t)$; the proofs of the statement about $Q_{\tM}^W(t)$ is 
nearly identical.  By Remark \ref{R}, it will be sufficient to prove that there is a polynomial $r_{k,h}^{\fS_h}(t)$ such that
$$R_{\tM}^W(t) = R_{M}^W(t) + \Ind_{\WH}^{W}\Res_{\WH}^{\fS_h}r_{k,h}^{\fS_h}(t);$$
the polynomials $p_{k,h}^{\fS_h}(t)$ and $z_{k,h}^{\fS_h}(t)$ can be obtained from $r_{k,h}^{\fS_h}(t)$ in the same way that we obtain $P_M^W(t)$ and $Z_M^W(t)$ from $R_M^W(t)$.
Assume $k>1$, and consider the difference 
$$R_{\tM}^W(t) - R_M^W(t) 
= \sum_{[S] \in (2^E\smallsetminus\{\emptyset\})/W} t^{\rk S} \Ind_{\WS}^W \left(P_{\tM_S}^{\WS}(t)-P_{M_S}^{\WS}(t)\right).$$
We break the sum into four different parts and analyze each part individually.
\begin{itemize}
\item We have $\tM_E = M_E$, so the summand indexed by $[E]$ vanishes.
\item Suppose $S$ is a proper subset of $E$ that is not contained in any element of $\cH$.  In this case, $\tM_S = M_S$, so the
summand indexed by $[S]$ vanishes.
\item The set $H$ is a flat of $M$ but not of $\tM$, and therefore $P_{\tM_H}^{\WH}(t) = 0$ by Remark \ref{loop}.
The contraction $M_H$ is uniform of rank 1, so $P_{M_H}^{\WH}(t) = \tau_{\WH}$.
Thus the summand indexed by $[H]$ is equal to 
$$- \Ind_{\WH}^W\tau_{\WH} = -\Ind_{\WH}^W\Res^{\fS_h}_{\WH} \tau_{\fS_h}.$$
\item Suppose that $\emptyset \subsetneq S \subsetneq H$.  
Our inductive hypothesis and Lemma \ref{claim1} tell us that the contribution indexed by $[S]$ is equal to 
$$t^{\rk S} \sum_{[T]\in \WH\backslash C(S,H)}\Ind_{\WT}^W\left( P_{\widetilde{M_T}}^{\WT}(t) - P_{M_T}^W(t)\right).$$
If we take the sum over all such $[S]$, we get
$$t^{\rk S}  \sum_{[T]\in (2^H\smallsetminus\{\emptyset,H\})/\WH}\Ind_{\WT}^W\left( P_{\widetilde{M_T}}^{\WT}(t) - P_{M_T}^W(t)\right).$$
Our inductive 
hypothesis tells us that
$$P_{\widetilde{M_T}}^{\WT}(t) - P_{M_T}^{\WT}(t) = \Ind^{\WT}_{\WT\cap\WH}\Res_{\WT\cap\WH}^{\fS_{h-|T|}} p_{k-|T|,h-|T|}^{\fS_{h-|T|}}(t),$$
and therefore that
\beq  \Ind_{\WT}^W\left( P_{\widetilde{M_T}}^{\WT}(t) - P_{M_T}^{\WT}(t)\right)
&=& \Ind_{\WT}^W \Ind^{\WT}_{\WT\cap\WH}\Res_{\WT\cap\WH}^{\fS_{h-|T|}} p_{k-|T|,h-|T|}^{\fS_{h-|T|}}(t)\\
&=& \Ind_{\WT\cap\WH}^W \Res_{\WT\cap\WH}^{\fS_{h-|T|}} p_{k-|T|,h-|T|}^{\fS_{h-|T|}}(t)\\
&=& \Ind_{\WH}^W\Ind_{\WT\cap\WH}^{\WH}\Res_{\WT\cap\WH}^{\fS_{h-|T|}} p_{k-|T|,h-|T|}^{\fS_{h-|T|}}(t).\eeq
Taking the sum over all
$[T]\in (2^H\smallsetminus\{\emptyset,H\})/\WH$, we get
\beq &&  \sum_{[T]\in (2^H\smallsetminus\{\emptyset,H\})/\WH}\Ind_{\WH}^W\Ind_{\WT\cap\WH}^{\WH}\Res_{\WT\cap\WH}^{\fS_{h-|T|}} p_{k-|T|,h-|T|}^{\fS_{h-|T|}}(t)\\
&=& \Ind_{\WH}^W\left( \sum_{[T]\in (2^H\smallsetminus\{\emptyset,H\})/\WH}\Ind_{\WT\cap\WH}^{\WH} \Res_{\WT\cap\WH}^{\fS_{h-|T|}} p_{k-|T|,h-|T|}^{\fS_{h-|T|}}(t)\right)\\
&=& \Ind_{\WH}^W\left( \sum_{\emptyset\subsetneq T\subsetneq H} \Res_{\WT\cap\WH}^{\fS_{h-|T|}} p_{k-|T|,h-|T|}^{\fS_{h-|T|}}(t)\right),
\eeq
where the second equality is a standard fact about induced representations; see for example \cite[lemma 2.7]{proudfootkls}.
(Note that the individual terms in the sum are not representations of $\WH$, but rather of $\WH\cap \WT$.  An element $w\in \WH$ takes the term indexed by $T$ to the term indexed by $wT$.)
We may rewrite this expression as
$$\Ind_{\WH}^W\Res^{\fS_h}_{\WH}\left( \sum_{\emptyset\subsetneq T\subsetneq [h]} p_{k-|T|,h-|T|}^{\fS_{h-|T|}}(t)\right),$$
where now the individual terms in the sum are representations of $\fS_{[h]\smallsetminus T}\cong \fS_{h-|T|}$, and the entire sum is a representation of $\fS_h$.
Finally, we once again employ the same standard fact about
induced representations, this time using the action of $\fS_h$, to rewrite our expression as
$$\Ind_{\WH}^W\Res^{\fS_h}_{\WH}\sum_{i=1}^{h-1}\Ind_{\fS_i\times\fS_{h-i}}^{\fS_h}\left(\tau_{\fS_i}\boxtimes 
p_{k-i,h-i}^{\fS_{h-i}}(t)\right),$$ which is manifestly of the desired form.
\end{itemize}
Putting the four parts together, we may take $$r_{k,h}^{\fS_h}(t) = - \tau_{\fS_h} + \sum_{i=1}^{h-1}\Ind_{\fS_i\times\fS_{h-i}}^{\fS_h}\left(\tau_{\fS_i}\boxtimes 
p_{k-i,h-i}^{\fS_{h-i}}(t)\right).$$
This completes the proof.
\end{proof}

\section{The second theorem}
This section is devoted to the proof of Theorem \ref{explicit}.  
The $k=1$ case was treated as the base case of the induction in the proof of Theorem \ref{main}, 
so we may assume that $k>1$.
Following the strategy of \cite{fnv2021}, we prove this result by examining a single example.
Let $$M_{k,h} := U_{k-1,h} \oplus B_1$$
be the direct sum of the uniform matroid of rank $k-1$ on $h$ elements and the Boolean matroid of rank 1.
The group $\fS_h$ acts on the first summand, which is a stressed hyperplane of cardinality $h$ \cite[Proposition 3.11]{fnv2021}.
The relaxation $\tM_{k,h}$ is equal to $U_{k,h+1}$.
We have the equalities
\beq p_{k,h}^{\fS_h}(t) &=& P_{\tM_{k,h}}^{\fS_h}(t) - P_{M_{k,h}}^{\fS_h}(t)\\
q_{k,h}^{\fS_h}(t) &=& Q_{\tM_{k,h}}^{\fS_h}(t) - Q_{M_{k,h}}^{\fS_h}(t),
\eeq
so it will suffice to compute the four polynomials on the right-hand sides of the two equations.

We begin with the polynomials associated with the matroid $M_{k,h}$.
We have
$$P_{B_1}^{\fS_h}(t) = Q_{B_1}^{\fS_h}(t) = V_{[h]},$$ 
and each of our three polynomials is multiplicative with respect to direct sums.
By \cite[Theorem 3.7]{GXY21-inverse-equivariant}, we have
\begin{equation}\label{pm}P_{M_{k,h}}^{\fS_h}(t) = P_{U_{k-1,h}}^{\fS_h}(t) = \sum_{i < (k-1)/2} V_{[h-2i,(k-2i)^i]/[(k-2i-2)^i]}\ t^i.\end{equation}
By \cite[Theorem 3.2]{GXY21-inverse-equivariant}, we have
\begin{equation}\label{qm}Q_{M_{k,h}}^{\fS_h}(t) = Q_{U_{k-1,h}}^{\fS_h}(t) = \sum_{i < (k-1)/2} V_{[h-k+2,2^i,1^{k-2i-2}]}\ t^i.\end{equation}
By the same theorems,
we have
\begin{equation}\label{pmt}P_{\tM_{k,h}}^{\fS_h}(t) = \Res_{\fS_h}^{\fS_{h+1}}P_{U_{k,h+1}}^{\fS_{h+1}}(t) = \sum_{i < (k-1)/2} \Res_{\fS_h}^{\fS_{h+1}}V_{[h-2i+1,(k-2i+1)^i]/[(k-2i-1)^i]}\ t^i\end{equation} and
\begin{equation}\label{qmt}Q_{\tM_{k,h}}^{\fS_h}(t) = \Res_{\fS_h}^{\fS_{h+1}}Q_{U_{k,h+1}}^{\fS_{h+1}}(t) = \sum_{i < (k-1)/2} \Res_{\fS_h}^{\fS_{h+1}}V_{[h-k+2,2^i,1^{k-2i-1}]}\ t^i.\end{equation}

We compute the restrictions using the following lemma.

\begin{lemma}\label{restrictions}
If $\la$ is a partition of $h+1$, then $$\Res_{\fS_h}^{\fS_{h+1}} V_\la = \bigoplus_{\la'}V_{\la'},$$
where $\la'$ ranges over partitions of $h$ with the property that the Young diagram for $\la'$ is obtained from the Young
diagram for $\la$ by removing a single box.
If $\la$ and $\mu$ are partitions with $|\la|-|\mu|=h+1$, then $$\Res_{\fS_h}^{\fS_{h+1}} V_{\la/\mu} = \bigoplus_{\mu'}V_{\la/\mu'},$$
where $\mu'$ ranges over partitions with the property that the Young diagram for $\mu'$ is obtained from the Young
diagram for $\mu$ by adding a single box.
\end{lemma}

\begin{proof}
The first statement is a well known special case of the Pieri rule.
To prove the second statement, let $\nu$ be any partition of $h$.  By Frobenius reciprocity, the multiplicity of $V_\nu$
in $\Res_{\fS_h}^{\fS_{h+1}} V_{\la/\mu}$ is equal to the dimension of the hom space from 
$\Ind_{\fS_h}^{\fS_{h+1}} V_\nu$ to $V_{\la/\mu}$, which is in turn equal to the multiplicity of $V_\la$ in
$$\Ind_{\fS_{|\nu|}\times\fS_1\times\fS_{|\mu|}}^{\fS_{h+1}} V_\nu\boxtimes V_{[1]}\boxtimes V_\mu.$$
By the Pieri rule, this may be reinterpreted as the sum over all $\mu'$ of the stated form of the multiplicity of $V_\la$ in
$$\Ind_{\fS_{|\nu|}\times\fS_{|\mu'|}}^{\fS_{h+1}} V_\nu\boxtimes V_{\mu'}.$$
In other words, it is the multiplicity of $V_\nu$ in $\bigoplus_{\mu'}V_{\la/\mu'}$.
\end{proof}

Applying the second statement of Lemma \ref{restrictions} to Equation \eqref{pmt}, we find that
\begin{equation}\label{pmt2}P_{\tM_{k,h}}^{\fS_h}(t) = \sum_{i < (k-1)/2} \Big(V_{[h-2i+1,(k-2i+1)^i]/[(k-2i-1)^i,1]}
+ V_{[h-2i+1,(k-2i+1)^i]/[k-2i,(k-2i-1)^{i-1}]}\Big)\ t^i,\end{equation}
where we interpret the second term inside the parentheses to be zero if $i=0$.  Similarly,
applying the first statement of Lemma \ref{restrictions} to Equation \eqref{qmt}, we find that
\begin{equation}\label{qmt2}Q_{\tM_{k,h}}^{\fS_h}(t) = \sum_{i < (k-1)/2} \Big(
V_{[h-k+2,2^i,1^{k-2i-2}]}
+ V_{[h-k+2,2^{i-1},1^{k-2i}]}
+ V_{[h-k+1,2^i,1^{k-2i-1}]}
\Big)\ t^i,\end{equation}
where we interpret the second term to be zero if $i=0$, and we interpret the third term to be zero if $i>0$ and $k=h$.

\begin{proof}[Proof of Theorem \ref{explicit}.]
The $k=1$ case was treated as the base case of the induction in the proof of Theorem \ref{main}, 
so we may assume that $k>1$.
We compute $p_{k,h}^{\fS_h}(t)$ by taking the difference between Equations \eqref{pmt2} and \eqref{pm}.
We observe that we have an isomorphism $$V_{[h-2i+1,(k-2i+1)^i]/[(k-2i-1)^{i},1]} \cong V_{[h-2i,(k-2i)^i]/[(k-2i-2)^i]}$$
of skew Specht modules, which follows from the fact that the skew diagrams $$[h-2i+1,(k-2i+1)^i]/[(k-2i-1)^{i},1]\and
[h-2i,(k-2i)^i]/[(k-2i-2)^i]$$ are related by a horizontal translation \cite[Proposition 2.3.5, Lemma 2.3.12]{kleshchev}.
This leads to a cancelation which gives us the formula for $p_{k,h}^{\fS_h}(t)$ stated in Theorem \ref{explicit}.
We compute $q_{k,h}^{\fS_h}(t)$ by taking the difference between Equations \eqref{qmt2} and \eqref{qm}.
\end{proof}

\section{The V\'amos matroid}
In this section we consider the {\bf V\'amos matroid} $V$, which is the smallest non-realizable matroid.  The ground set of $V$ is equal to $[8]$,
and it is a paving matroid of rank 4 with 5 circuit-hyperplanes corresponding to the five shaded rectangles in Figure 1.
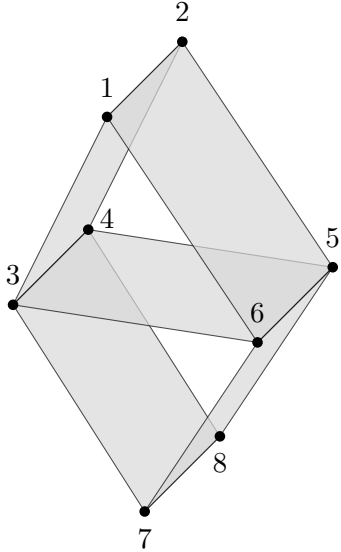
\begin{figure}[h]
\begin{center}
\begin{tikzpicture}
\draw[fill=gray!30,opacity=.7]    (0,0) -- (1,1) -- (-.25,-1.5)--(-1.25,-2.5)--(0,0);
\draw[fill=gray!30,opacity=.7]    (.5,-5.25) -- (1.5,-4.25) -- (-.25,-1.5)--(-1.25,-2.5)--(.5,-5.25);
\draw[fill=gray!30,opacity=.7]    (.5,-5.25) -- (1.5,-4.25) --(3,-2)-- (2,-3)--(.5,-5.25);
\draw[fill=gray!30,opacity=.7]    (-.25,-1.5)--(-1.25,-2.5) -- (2,-3)--(3,-2)--(-.25,-1.5);
\draw[fill=gray!30,opacity=.7]    (0,0) -- (1,1) -- (3,-2)--(2,-3)--(0,0);

\foreach \Point in {(0,0), (1,1), (.5,-5.25), (1.5,-4.25), (3,-2), (2,-3),(-.25,-1.5),(-1.25,-2.5)}{
    \fill \Point circle[radius=2pt];

}
    \node[label={$1$}] (a) at (0,0) {};
    \node[label={$2$}] (a) at (1,1) {};
    \node[label={$3$}] (a) at (-1.25,-2.5) {};
    \node[label={$4$}] (a) at (0,-1.75) {};
    \node[label={$5$}] (a) at (3,-2) {};
    \node[label={$6$}] (a) at (2,-3) {};
    \node[label={$7$}] (a) at (.5,-6) {};
    \node[label={$8$}] (a) at (1.5,-5) {};
\end{tikzpicture}
\end{center}
\caption{A visual representation of the circuit-hyperplanes for $V$.}\label{fig:vamos}
\end{figure}
The automorphism group $W$ of $V$ is generated by the following four elements:
$$r_1 = (12),\quad s_1 = (17)(28), \quad r_2 = (34), \quad\text{and}\quad s_2 = (35)(46).$$
Note that $W \cong D_4\times D_4$, where the first factor is generated by $r_1$ and $s_1$ and the second factor is generated by $r_2$ and $s_2$.

Let $H := \{1,2,3,4\}$ and $H' := \{3,4,5,6\}$.  The orbit of $H$ under the action of $W$ consists of the four circuit hyperplanes other than $H'$, and the
stabilizer of $H$ is $$\WH = \langle (12), (34), (56), (78)\rangle \cong\fS_2^4.$$
In contrast, $H'$ is fixed by $W$.  By Theorem \ref{main}, we have
$$P_V^W(t) = \Res^{\fS_8}_W P_{U_{4,8}}^{\fS_8}(t) - \Ind_{\WH}^W\Res^{\fS_4}_{\WH} p_{4,4}^{\fS_4}(t) - \Res^{\fS_4}_W p_{4,4}^{\fS_4}(t).$$
Here the first restriction is the pullback along the homomorphism from $\WH$ to $\fS_4$ given by the action of $\WH$ on $H\cong [4]$,
while the second is the pullback along the homomorphism from $W$ to $\fS_4$ given by the action of $W$ on $H'\cong [4]$.

The formula for $P_{U_{4,8}}^{\fS_8}(t)$
is given in \cite[Theorem 3.1]{gpy-equi-kl} or \cite[Theorem 3.7]{GXY21-inverse-equivariant}, and the formula for $p_{4,4}^{\fS_4}(t)$ is given in Theorem \ref{explicit}.
Note that the constant term of $P_{U_{4,8}}^{\fS_8}(t)$ is equal to the trivial representation of dimension 1, as is the case for all loopless matroids \cite[Corollary 2.10]{gpy-equi-kl}.
All three polynomials are linear, so the only nontrivial calculation is of the coefficient of $t$.

The calculation can be done explicitly using character tables.  We use the following standard representation of the character table for $D_4$:
\begin{center}
\begin{tabular}{|c|c|c|c|c|c|}
\hline  & $e$ & $s$ & $r^2$ & $sr$ & $r$\\ \hline
$\chi_1$& $1$& $1$ & $1$ & $1$ & $1$\\\hline
$\chi_2$& $1$& $1$ & $-1$ & $1$ & $-1$\\\hline
$\chi_3$& $1$ & $-1$ & $-1$ & $1$ & $1$\\\hline
$\chi_4$& $1$&$-1$ & $1$ & $1$ &$-1$\\\hline
$\chi_5$& $2$& $0$ & $0$ & $-2$ & $0$\\\hline
\end{tabular}
\end{center}
The irreducible characters of $W\cong D_4\times D_4$ are of the form $\chi_i\boxtimes\chi_j$ for $i,j\in\{1,\ldots,5\}$.
After performing all of the restrictions and inductions,
we find that the character of the linear term of $P_M^W(t)$ is equal to
\begin{align*}
3\chi_1\boxtimes \chi_1+\chi_1\boxtimes\chi_2	+ \chi_1\boxtimes\chi_4+2\chi_2\boxtimes\chi_1 +\chi_2\boxtimes\chi_2+\chi_2\boxtimes\chi_4+\chi_4\boxtimes\chi_1+\chi_4\boxtimes\chi_2\\+\chi_1\boxtimes\chi_5+\chi_2\boxtimes\chi_5+\chi_4\boxtimes\chi_5+2\chi_5\boxtimes\chi_1+\chi_5\boxtimes\chi_2+\chi_5\boxtimes\chi_5+2\chi_5\boxtimes\chi_5.
\end{align*}
We observe that the value of this character on the identity is 33, so the non-equivariant Kazhdan--Lusztig polynomial of $V$ is $P_V(t) = 1 + 33t$.

\section{Steiner systems}
A {\bf Steiner system} of type $(d,k,n)$ consists of a set $E$ of cardinality $n$ along with a family $\cH$ of $k$-element subsets (called {\bf blocks}) with the 
property that every $d$-element subset of $E$ is contained in exactly one block.  
A Steiner system $(E,\cH)$ of type $(d,k,n)$ determines a paving matroid of rank $d+1$ on the ground set $E$ 
characterized by the property that $\cH$ is the set of hyperplanes \cite[Chapter 12.3]{welsh}.
Given a Steiner system $(E,\cH)$ of type $(d,k,n)$ and an element $e\in E$, 
one can construct a new Steiner system $(E/e, \cH/e)$ of type $(d-1,k-1,n-1)$ by putting $E/e := E\smallsetminus\{e\}$ and $\cH/e := \{H\smallsetminus\{e\}\mid e\in H\in\cH\}$.

There is a unique Steiner system of type $(5,6,12)$ up to isomorphism, which is typically denoted $S(5,6,12)$.  The automorphism group of $S(5,6,12)$ is the Mathieu group $M_{12}$.
This group acts 4-transitively on the ground set, and the stabilizer of a point is the Mathieu group $M_{11}$.  Thus we may perform the aforementioned operation to obtain a Steiner
system $S(4,5,11)$ with an action of $M_{11}$.

There is also a unique Steiner system of type $(5,8,24)$ up to isomorphism, which is denoted $S(5,8,24)$, and is known as the {\bf Witt geometry}.
The automorphism group of $S(5,8,24)$ is the Mathieu group $M_{24}$, which 
acts 5-transitively on the ground set.  The stabilizer of a single point is the Mathieu group $M_{23}$, which acts on the corresponding Steiner system $S(4,7,23)$.
The stabilizer of a pair of points is the Mathieu group $M_{22}$, which acts on the corresponding Steiner system $S(3,6,22)$.
The Mathieu groups $M_{11}$, $M_{12}$, $M_{22}$, $M_{23}$, and $M_{24}$ are all sporadic finite simple groups.

\begin{remark}
The Mathieu groups $M_{11}$, $M_{12}$, $M_{23}$, and $M_{24}$ are each equal to the automorphism groups of their corresponding Steiner systems.
In contrast, $M_{22}$ is the unique index 2 subgroup of the automorphism group of $S(3,6,22)$.
\end{remark}

We will use the same notation to refer to a Steiner system and its associated matroid.  For example, we will denote by
$P_{S(5,8,24)}^{M_{24}}(t)$ the $M_{24}$-equivariant Kazhdan--Lusztig polynomial of the matroid associated with the Steiner
system $S(5,8,24)$.  We will refer to irreducible characters of the Mathieu groups by the same indices used in the {$\Bbb{ATLAS}$} of Finite  Groups \cite{atlas}.

\begin{proposition}
The equivariant Kazhdan--Lusztig polynomials of the matroids associated with the aforementioned Steiner systems are characterized
as follows:
\beq
\operatorname{char} P_{S(4,5,11)}^{M_{11}}(t) &=& \chi_1 + \left(\chi_5+ \chi_8 \right) t + \left(\chi_5+ \chi_8 \right)t^2\\
\operatorname{char} P_{S(5,6,12)}^{M_{12}}(t) &=& \chi_1
    + \left(\chi_3+ \chi_7+ \chi_8 \right) t
    +\left(\chi_3+ \chi_7+ \chi_8+ \chi_{11}+ \chi_{12}+ \chi_{14} \right)t^2
 \\
\operatorname{char} P_{S(3,6,22)}^{M_{22}}(t) &=& \chi_1 + \chi_5\, t\\
\operatorname{char} P_{S(4,7,23)}^{M_{23}}(t) &=& \chi_1 + \chi_5\, t + \chi_9\, t^2\\
\operatorname{char} P_{S(5,8,24)}^{M_{24}}(t) &=& \chi_1 +  \left(\chi_8 + \chi_9 \right) t + \left(\chi_9+\chi_{14}+ \chi_{21} \right) t^2.
\eeq
Non-equivariantly, we have
\beq
P_{S(4,5,11)}(t) &=& 1 + 55t + 55t^2\\
P_{S(5,6,12)}(t) &=& 1 +120t+429t^2\\
P_{S(3,6,22)}(t) &=& 1 + 55t\\
P_{S(4,7,23)}(t) &=& 1 +230t+253t^2\\
P_{S(5,8,24)}(t) &=& 1+735t+4830t^2.
\eeq
\end{proposition}

\begin{proof}
All of these calculations are done using only Theorems \ref{main} and \ref{explicit} along with the character tables found in the
{$\Bbb{ATLAS}$}.  We provide a brief outline of the calculation only for the most interesting case, namely $S(5,8,24)$.

The ground set of the matroid $S(5,8,24)$ is $\{1,\ldots,24\}$. 
The group $M_{24}$ acts transitively on the set of blocks.
We have a distinguished block $H = \{1,\ldots,8\}$, whose stabilizer group is isomorphic to
$\mathfrak{A}_8 \ltimes \mathbb{F}_2^4$, where the alternating group $\mathfrak{A}_8 \cong \operatorname{GL}_4(\mathbb{F}_2)$
acts linearly on the vector space $\mathbb{F}_2^4$.  The homomorphism from the stabilizer group to 
$\fS_H \cong \fS_8$ is given by the projection
onto $\mathfrak{A}_8$ followed by the inclusion of $\mathfrak{A}_8$ into $\fS_8$.
Theorem \ref{main} tells us that
$$P_{S(5,8,24)}^{M_{24}}(t) = \Res_{M_{24}}^{\fS_{24}} P_{U_{6,24}}^{\fS_{24}}(t) - \Ind_{\mathfrak{A}_8 \ltimes \mathbb{F}_2^4}^{M_{24}}\Res^{\fS_8}_{\mathfrak{A}_8 \ltimes \mathbb{F}_2^4} p_{6,8}^{\fS_8}(t).$$
Using the formula for $P_{U_{6,24}}^{\fS_{24}}(t)$ in  \cite[Theorem 3.1]{gpy-equi-kl} or \cite[Theorem 3.7]{GXY21-inverse-equivariant} and the formula for $p_{4,4}^{\fS_4}(t)$ given in Theorem \ref{explicit}, this becomes a straightforward (if cumbersome) computer computation.
\end{proof}

\bibliography{bibliography}
\bibliographystyle{amsalpha}

\end{document}